\documentclass[11pt]{amsart}
\usepackage{amssymb}
\usepackage{amsmath}
\usepackage{amsthm}
\usepackage{braket}
\usepackage{pdfpages}
\theoremstyle{plain}
\newtheorem{theorem}{Theorem}[section]

\theoremstyle{theorem}

\newtheorem{lem}[theorem]{Lemma}

\theoremstyle{definition}
\newtheorem{defn}[theorem]{Definition}
\newtheorem{rmk}[theorem]{Remark}

\newtheorem{cl1}{Claim}[]
\newtheorem{cl2}{Claim}[]

\newcommand{\m}{M\"{o}bius }
\newcommand{\R}{{\mathbb R}}

\title{ Weakly Circle-Preserving Maps in Inversive Geometry}		
\author{Joel C. Gibbons}
\address{Logistic Research \& Trading Co., P.O. Box 63, St. Joseph, MI 49085}
\author{ Yusheng Luo}
\address{Dept. of Mathematics, National Univ. of Singapore, 10 Lower Kent Ridge Road, Singapore 119076}
\date{Aug 8, 2013, v.arXiv}

\begin{document}

\begin{abstract}
  Let $\mathbb{S}^n$ be the standard $n$-sphere embedded in $\R^{n+1}$. A mapping $T: \mathbb{S}^n \to \mathbb{S}^n$,
not assumed continuous or even measurable, nor injective,  is called weakly circle-preserving if the image of any 
circle under $T$  is contained in some circle  in the range space $\mathbb{S}^n$. The main result of this paper shows that any weakly circle-preserving map satisfying a very mild condition on its range $T(\mathbb{S}^n)$ must be a \m transformation. 
\end{abstract}

\maketitle
\section{Introduction}

The object of this paper is to give a   characterizations of \m transformations acting on ${\mathbb S}^n$,
under very weak conditions on such a map $T: {\mathbb S}^n \to {\mathbb S}^n$, 
which do not assume invertibility or even continuity of the map.

The  standard $n$-sphere ${\mathbb S}^n$, viewed in $\mathbb{R}^{n+1}$  is the real algebraic set 
$$
\mathbb{S}^n := \{ (x_0, x_1,  \cdots , x_{n})\in \R^{n+1} : \sum_{i=0}^{n} x_i^2 =1\}.
$$

The set of \m transformations are the set of invertible maps $F : \mathbb{S}^n \to \mathbb{S}^n$
generated by inversions. Such maps send circles to circles and $(n-1)$-spheres to $(n-1)$-spheres.
The study of geometric properties invariant under such transformations is called inversive geometry.

One can also identify $\mathbb{S}^n$ with $\R^n_\infty:=\R^n\cup\{\infty\}$ under sterographic projection.
In the space $\R^n_{\infty}$, an \begin{em} inversion \end{em}(or a \begin{em}reflection\end{em}) in an $(n-1)$-sphere $S(a,r):=\{x\in\R^n:|x-a| = r\}$ is the function $\phi$ defined by $\phi(x) = a+(\frac{r}{|x-a|})^2(x-a)$. $\phi$ is well defined on $\R^n_\infty - \{a,\infty\}$, and at these two points, we define $\phi(a) = \infty$ and $\phi(\infty) = a$. A \begin{em}reflection\end{em} in a hyperplane 
is a usual reflection in $\R^n$ and fixes the point $\infty$ in $\R^n_{\infty}$.
We define  a {\em \m transformation} on $\R^n_\infty \cong \mathbb{S}^n$ to be a finite composition of reflections in 
$(n-1)$-spheres or hyperplanes. The group of all \m transformations is called 
the Generalized \m Group $GM(\R^n_\infty)$, following Beardon \cite[Chapter 3]{B83}.
Note that in dimension $n=2$, identifying $\R^2_{\infty}$ with the Riemann sphere $\hat{\mathbb{C}}$, \m transformations include all
 the linear fractional transformations $z\mapsto \frac{az+b}{cz+d}$,  which are orientation-preserving maps 
 and also the conjugate ones $z\mapsto\frac{a\bar z+b}{c\bar z+d}$,
  which are orientation reversing maps.

\subsection{Main Result}
We will study mappings satisfying the following very  weak version of
the circle-preserving property. 

\begin{defn}\label{def-11}
A map $T$ of the $n$-sphere to itself is called {\em weakly circle-preserving} if for every circle $C \subset \mathbb{S}^n$, $T(C)$ lies in some circle.
\end{defn}

\begin{defn} 
A map $T$ of the $n$-sphere to itself is called {\em weakly sphere-preserving} if for every $(n-1)$-sphere $S_{n-1} \subset \mathbb{S}^n$, $T(S_{n-1})$ lies in some $(n-1)$-sphere.
\end{defn}

In these two definitions we do not assume  that $T$ is injective or
even  continuous.  There are many such maps, including some that are
not \m transformations.
For example, any map $T: \mathbb{S}^n \to \mathbb{S}^n$ whose image is finite and consists of $(n+1)$ points or less is automatically weakly sphere-preserving and any map on $n$-sphere with image consisting of $3$ points or less is weakly circle-preserving. 
Nevertheless such maps are quite restricted when further assumptions are imposed on them. 

The key restrictions we consider are the following ``general position" conditions on the image of the map.

\begin{defn}
A subset $B$ of $\mathbb{S}^n$ is said to lie in {\em circular general position} if for any circle $C$, 
the complement of $C$ contains at least two points of $B$.
\end{defn}

\begin{defn}\label{def-12}
A subset $B$ of $\mathbb{S}^n$ is said to lie in {\em spherical general position} if for any $(n-1)$-sphere $S_{n-1}$, the complement of $S_{n-1}$ contains at least two points of $B$.
\end{defn}

It is obvious from the definition that $B$ must contain at least $n+3$ points,
and there do exist many $(n+3)$ point sets in spherical
general positon.  If $B$ lies in spherical general position, 
then every set containing $B$ lies in spherical general position. However, if $B$ lies in spherical general position, it
is not clear whether it always contains an  $n+3$ subset of $B$ lies in spherical general position,
and this result does always not hold in dimension $n=2$.

Our main result is as follows.

\begin{theorem}\label{thm-main}
For $n\geq 3$, $T:\mathbb{S}^n\longrightarrow \mathbb{S}^n$ be a weakly circle-preserving map. If $T(\mathbb{S}^n)$ is in spherical general position and there is a $2$-sphere $S_2$ with $T(S_2)$ in circular general position, then $T$ is a \m transformation.
\end{theorem}

This result strengthens many previous characterizations of \m transformations, which we discuss below. 
We remark that the notion of  the weakly-circle
preserving,  in dimensions $2$ and $3$,  was first studied by the first author  \cite{GW79}, \cite{Gib82},
where it was termed  ``circle-preserving".

\subsection{Previous Results}
Rigidity theorems of \m transformation have been investigated extensively. 
It is clear  that \m transformations take generalized $(n-1)$-spheres
to generalized $(n-1)$-spheres  and a converse was known to M\"{o}bius,
under the assumption that the map $T$ is continuous, 
see Blair \cite[Theorem 5.6]{Bla00}.
A map is called conformal if the map preserves  angles. One can define it formally using conformal manifolds, 
see (Kobayashi \cite{Kob95}, Blair \cite{Bla00}). 
Any \m transformation is a conformal diffeomorphism on $\mathbb{S}^n$. A result of  Liouville\cite{Lio1850}  in 1850 
 asserts (in modern form) a  local converse: when $n\geq 3$, any smooth conformal diffeomorphism of a simply
 connected open domain $U$ of ${\mathbb S}^n$ into ${\mathbb S}^n$ is 
 the restriction of a \m transformation. In particular, for $n\geq 3$, the conformal group on $\mathbb{S}^n$ 
 is precisely the generalized \m group.

In 1937, Carath\'eodory \cite{Car37} proved that  
a local version of the circle-preserving condition is enough to
force a map to be (part of)
a \m transformation.
Given a domain $U \subset \R^2$, and a 1-1 map $T:U\longrightarrow \R^n$ with $n\geq 2$, such that 
fpr any circle $C$ in $U$ that is contractible in $U$, $T(C)$ is a circle,  then $T(U)$ lies inside a plane 
and $T$ is a restriction of a \m transformation in $\R^2_{\infty}$.

 More recently in 2001, Beardon and Minda \cite{BM01} proved
 that $T:\mathbb{S}^n\longrightarrow\mathbb{S}^n$ is a \m transformation if and only if $T$ locally maps 
 each $(n-1)$-sphere onto $(n-1)$-sphere.
 In 2005  Li and Wang \cite{LW05} showed that $T:\mathbb{S}^n \longrightarrow \mathbb{S}^n$ is a \m transformation 
 if and only if $T$ is circle preserving and $T(\mathbb{S}^n)$ is not a circle. 
 All the above characterizations assume that $T$ is circle (or sphere) preserving, i.e., 
 $T$ maps circles (or spheres) onto circles (or spheres). 

The  weakly circle preserving assumption is  much
less restrictive than those assumed on maps above.
This assumption was introduced in 1979 by the first author with Webb \cite{GW79}.
Unfortunately, that paper used  the term ``circle-preserving map" 
to mean `` weakly circle-preserving map" as above. The paper \cite{GW79} established a local result, under the weakly circle preserving map hypothesis.  A special case  is stated below  in Theorem \ref{six-thm}.

In this paper we prove only a global result, assuming the map $T$ is defined on all of $\mathbb{S}^n$.
However we expect that the main result extends to a local version, where one assumes only that
$T$ is defined on a simply connected open set inside $\mathbb{S}^n$.

\subsection{Notation}
Given $A \subset \mathbb{S}^n$, we will denote $(A)$ to be the smallest dimension sphere in $\mathbb{S}^n$ that contains $A$. If $A$ is some finite set, e.g., $A = \{x_1,...,x_n\}$, we will simply write $(x_1x_2...x_n)$ to mean $(\{x_1,x_2,...,x_n\})$. Similarly, if $A = B\cup\{x_1,x_2,...,x_n\})$, we will simply use $(B,x_1,x_2,...,x_n)$ to mean $(B\cup\{x_1,x_2,...,x_n\})$.

To distinguish the domain and range space, we will use lower case roman letter (e.g., $x$) to denote a point in the domain space and lower case roman letter with an apostrophe symbol (e.g., $x'$) to denote a point in the range space.

We will also use $S_k$ to denote $k$-sphere in $\mathbb{S}^n$ and if $T$ is a map on $\mathbb{S}^n$, we will use the notation $S_k':=(T(S_k))$, i.e., the smallest dimension sphere containing $T(S_k)$.

Throughout this paper, we will identify $\mathbb{S}^n$ with $\mathbb{R}^n_\infty:=\mathbb{R}^n\cup\{\infty\}$, and a $k$-sphere with either Euclidean 
$k$-sphere in $\mathbb{R}^n$ or a k-dimensional affine space in $\mathbb{R}^n$ together with the point $\infty$. We will also often use $\R^k_\infty \subset \R^n_\infty$ to mean the subspace $\{(x_1,...,x_k,0,...,0):x_1,...,x_k\in \R\}\cup \{\infty\}$ in $\R^n_\infty$.

\section{Two dimensional case }\label{sec2}

The proof will use some results from the two-dimensional case which were obtained in previous papers,
\cite{GW79} and \cite{GL13a}.
They state that if $T$ is weakly circle-preserving, 
and satisfies  some conditions on the image of the map $T$, 
then $T$ will automatically become continuous and bijective, in fact, $T$ will be a \m transformation.

In 1979 the first author with Webb \cite[Theorem 1]{GW79} proved a  "six-point theorem" for locally defined maps.
The following 
theorem  is the special case $U= {\mathbb S}^2$ of that theorem. 
Notice that we use the term \m transformation for inversive transformation.
\begin{theorem} \label{six-thm}   {\em (``Six-point theorem")}
Let  $T$ be a weakly circle-preserving map from  ${\mathbb S}^2$ into  ${\mathbb S}^n$ with $n\geq 2$
which satisfies the following conditions. 
\begin{enumerate}
\item
Every circle in the codomain ${\mathbb S}^n$ does not contain
at least two points in the image $T({\mathbb S}^2)$, i.e.  
$T({\mathbb S}^2)$ is in circular general position in $\mathbb{S}^n$. 

\item
The image $T({\mathbb S}^2 )$ contains at least six distinct points. 
\end{enumerate}
Then $T({\mathbb S}^2 )$ is a 2-sphere and $T$ is a \m transformation.
\end{theorem}

A result that we proved in  \cite{GL13a} allows us to strengthen the result above, as follows. 
\begin{theorem}\label{5point}
 {\em (``Five-point theorem")}
Let  $T$ be a weakly circle-preserving map from  ${\mathbb S}^2$ into  ${\mathbb S}^n$ with $n\geq 2$
which satisfies the following conditions. 
\begin{enumerate}
\item
Every circle in the codomain ${\mathbb S}^n$ does not contain
at least two points in the image  $T({\mathbb S}^2)$, i.e.
$T({\mathbb S}^2)$ is in circular general position in $\mathbb{S}^n$. 
\item
The image $T({\mathbb S}^2 )$ contains five or more   distinct points. 
\end{enumerate}
Then $T({\mathbb S}^2 )$ is a 2-sphere and $T$ is a \m transformation.
\end{theorem}
%
%
\begin{rmk}
We have stated this theorem to be in parallel with Theorem \ref{six-thm}. 
In fact   condition (2) in Theorem \ref{5point}  already follows from condition (1), 
\end{rmk}

\begin{proof}
This was proved in \cite[Theorem 4.4]{GL13a}.
\end{proof}


\section{Proof of Main Theorem \ref{thm-main}.}

Before proving this theorem, we need to prove several lemmas. The first two lemmas are geometric facts about intersecting spheres, which will be used a lot in proving the other lemmas.

\begin{lem}\label{dim-sphere-intersect}
If $S_k$ and $S_m$ are $k$-sphere and m-sphere in $\mathbb{S}^n$, assume $S_k\cap S_m$ contains at least two points, then dimension of the sphere $S_k\cap S_m$ can be $\max\{0,k+m-n\}$ to $\min\{k,m\}$.
\end{lem}
\begin{proof}
To see this, we simply use a \m transformation to map two intersection points to $0$ and $\infty$, then $S_k$ and $S_m$ are mapped to two vector subspace in $\mathbb{R}^n$ union the point $\infty$. We apply basic dimension theorem in linear algebra and get the result.
\end{proof}

\begin{lem}\label{circle-in-S2}
For $n\geq 3$, $1\leq k\leq n-2$, let $S_k$ be a $k$-sphere in $\mathbb{S}^n$ and $x_1,x_2$ be two points in $\mathbb{S}^n$. Assume that $\dim((S_k,x_1,x_2)) = k+2$, then any circle $C$ through $x_1,x_2$ intersects $S_k$ in at most one point.
\end{lem}
\begin{proof}
Let $S_{k+1,1}:=(S_k,x_1)$, and suppose for contradiction that there is a circle $C$ through $x_1,x_2$ that intersect $S_k$ at at least two points. Let $y_1,y_2$ be two of the intersection points. Notice that $\{x_1,y_1,y_2\} \subset C\cap S_{k+1,1}$, so $C \subset S_{k+1,1}$. So we have $x_2\in S_{k+1,1}$. This is a contradiction to $\dim((S_k,x_1,x_2)) = k+2$.
\end{proof}

The following lemma shows that if $T$ is a \m transformation on a small dimensional sphere, $T$ will be \m transformation in a larger dimensional sphere provided some condition on the image of this larger dimensional sphere. This lemma gives us a tool to build up a chain of spheres with increasing dimension and $T$ is a \m transformation on each of them.

\begin{lem}\label{2point-lemma}
Let $T:\mathbb{S}^{n}\longrightarrow \mathbb{S}^{n}$ be a weakly circle preserving map, and $S_k^0\subset S_{k+1}^0$ are $k$-sphere and $(k+1)$-sphere in $\mathbb{S}^n$ respectively. Suppose that $T|_{S_k}$ is a \m transformation and $|T(S_{k+1}^0)-T(S_k^0)|\geq 2$, then $T|_{S_{k+1}^0}$ is a \m transformation.
\end{lem}
\begin{proof}
We will first show $T$ is injective on $S_{k+1}^0$. Let $x_1', x_2'\in T(S_{k+1})-T(S_k)$, and fix two points $x_i\in T^{-1}(x_i')$.

\begin{cl1}
There is a 2-sphere $S_2$ through $x_1$, $x_2$ and intersecting $S_k^0$ in a circle.
\end{cl1}
\begin{proof} (of Claim 1)

Since $S_k^0$ has codimension 1 in $S_{k+1}^0$, $S_k^0$ divides $S_{k+1}^0$ into two components.

If $x_1$ and $x_2$ are on the opposite side of $S_k^0$, then any 2-sphere through $x_1, x_2$ will intersect $S_k^0$ in a circle.

Otherwise, choose $x_3$ on the opposite side of $x_1$ and $x_2$, then any 2-sphere through $x_1,x_2,x_3$ will intersect $S_k^0$ in a circle.
\end{proof}

Notice that the image of $S_2$ contains $x_1',x_2'$ and a whole circle, so $T(S_2)$ is in circular general position. By the five-point theorem (Theorem \ref{5point}), $T$ is a \m transformation on $S_2$. Also notice that $S_2$ must intersect both components in $S_{k+1}^0$ divided by $S_k^0$, so we let $x_3 \in S_2-S_k^0$ on the opposite side of $x_1$, and $x_3'$ be its image. Since $T$ is a \m transformation on $S_2$, we have $x_3'\notin T(S_k^0)$ and $x_3'\neq x_1'$. Now given any two points $y_1,y_2$, we form a 2-sphere through $x_1,x_3,y_1,y_2$, then this 2-sphere must intersect $S_k^0$ in a circle as it contains $x_1$ and $x_2$. Therefore, the image of this 2-sphere consists of $x_1',x_3'$ and a whole circle, which means the image is in circular general position. Therefore, $T$ is a \m transformation on this 2-sphere by five-point theorem (Theorem \ref{5point}), in particular, $T(y_1)\neq T(y_2)$. This shows that $T$ is injective on $S_{k+1}^0$.\\

We will then prove that $T(S_{k+1}^0)$ lies in some $(k+1)$-sphere. Suppose not, let $x_1,x_2$ be two points in $S_{k+1}^0$ such that $x_2'\notin ((S_k^0)',x_1')$. Let $S_2$ be a 2-sphere through $x_1, x_2$ and intersecting $S_k^0$ in a circle $C$, then $T(S_2)$ is in circular general position, so $T$ is a \m transformation on $S_2$. But $S_2'\cap ((S_k^0)',x_1')$ contains $x_1'$ and $C'$, so $S_2'\subset ((S_k^0)',x_1')$, then $x_2'\in S_2'\subset ((S_k^0)',x_1')$. But this is a contradiction to the assumption $x_2'\notin ((S_k^0)',x_1')$. Therefore $T(S_{k+1})$ lies in some $(k+1)$-sphere.\\

We will now prove this lemma by induction.

The base case $k=2$: Under a \m transformation, we may assume that $S_3^0=\R^3_\infty\subset \R^n_\infty$, and $S_2^0=\R^2_\infty\subset \R^n_\infty$. Given any 2-sphere $S$ which is 2-dimensional plane through origin with $\infty$ in $S_3^0$ other than $S_2^0$, then $S\cap S_2^0$ is a circle $C$, so the image $T(S)$ contains a whole circle $T(C) = (T(C)) = C'$. Moreover, since $T$ is injective on $S_3^0$, $T(S)$ also consists of 2 distinct points off $C'$, so $T(S)$ is in circular general position. By the five-point theorem (Theorem \ref{5point}), $T$ is a \m transformation on $S$. Now given arbitrary 2-sphere $S_2$ in $S_3^0$, it must intersect some plane through origin in a circle. By the same argument, $T$ is a \m transformation on $S_2$. By the corollary 6.3 in Beardon and Minda \cite{BM01} (namely, let $T:\mathbb{S}^n\longrightarrow \mathbb{S}^n$, if $T$ restrict to any $(n-1)$-sphere is a \m transformation, then $T$ is a \m transformation), we conclude that $T$ is a \m transformation on $S_3^0$.

The induction step: Assume that the lemma holds for $k$, we prove the case for $k+1$. Again, under a \m transformation, we may assume that $S_{k+2}^0=\R^{k+2}_\infty\subset \R^n_\infty$, and $S_{k+1}^0=\R^{k+1}_\infty\subset \R^n_\infty$. Given any $(k+1)$-sphere $S$ which is (k+1)-dimensional plane through origin with $\infty$ in $S_{k+2}^0$ other than $S_{k+1}^0$, $S\cap S_{k+1}^0$ is a $k$-sphere $\tilde{S}$. Notice $T|_{\tilde{S}}$ is a \m transformation as $T$ is a \m transformation on $S_{k+1}^0$. Moreover, since $T$ is injective on $S_{k+2}^0$, $T(S)$ also consists of 2 distinct points off $T(\tilde{S})$. By induction hypothesis, $T$ is a \m transformation on $S$. Now given arbitrary $(k+1)$-sphere $S_{k+1}$ in $S_{k+2}^0$, it must intersect some hyperplanes in $S_{k+2}^0$ through origin in a $k$-sphere. By the same argument, $T$ is a \m transformation on $S_{k+1}$. By the corollary 6.3 in Beardon and Minda \cite{BM01}, we conclude that $T$ is a \m transformation on $S_{k+2}^0$.

By induction, we conclude that the lemma holds for every $k\geq 2$.
\end{proof}

The following lemma shows that if the hypothesis of the Lemma \ref{2point-lemma} fails for every possible $(k+1)$-spheres, there will be some strict restriction on all higher dimensional spheres.

\begin{lem}\label{1point-lemma}
Let $T:\mathbb{S}^n\longrightarrow \mathbb{S}^n$ be a weakly circle-preserving map and $S_k^0$ be a fixed $k$-sphere in $\mathbb{S}^n$ such that $T|_{S_k^0}$ is a \m transformation. Assume that for any $(k+1)$-sphere $S_{k+1}$ containing $S_k^0$, $|T(S_{k+1})-T(S_k^0)|\leq 1$, then the following holds:

For any $(k+m)$-sphere $S_{k+m}$ ($1\leq m\leq n-k$) containing $S_k^0$ such that $\dim(S_{k+m}') \geq k+m$ (recall here $S_{k+m}'$ is the smallest dimension sphere containing $T(S_{k+m})$), we have
\begin{equation}
|T(S_{k+m})-T(S_k^0)| = m
\end{equation}
and
\begin{equation}
|T(S_{k+m}-S_k^0)|= m
\end{equation}
\end{lem}
\begin{proof}
Notice that from the hypothesis $\dim(S_{k+m}') \geq k+m$, it is necessary that $|T(S_{k+m})-T(S_k^0)| \geq m$. Hence, to prove (1), it is sufficient to prove $|T(S_{k+m})-T(S_k^0)| \leq m$. The proof of this lemma is by induction from $m = 2$. However, the base cases for $m=1,2$ are treated differently.\\

Case $m = 1$: 

(1): This is immediate from the hypothesis of the lemma. \\

(2): Let $T(S_{k+1})-T(S_k^0) = \{x'\}$ and fix $x\in T^{-1}(x')$. Suppose for contradiction that there is $y\in S_{k+1}-S_k^0$ with $y':=T(y) \in T(S_k^0)$.

\begin{cl2}
There is a circle $C$ through $x,y$ and intersect $S_k^0$ at two points and $y'\notin T(C\cap S_k^0)$.
\end{cl2}
\begin{proof} (of Claim 1)

Since $S_k^0$ has codimension 1 in $S_{k+1}$, $S_k^0$ divides $S_{k+1}$ into two components.

If $x,y$ are in opposite side of $S_k^0$, then any circle containing $x,y$ intersects $S_k^0$ at two points. Let $C_1$ and $C_2$ be two different circles through $x,y$, then $C_1\cap C_2 \cap S_k^0 = \emptyset$ (as three points determines a circle). Since $T$ is injective on $S_k^0$, so $T(C_1\cap S_k^0)\cap T(C_2\cap S_k^0) = \emptyset$. Hence, $y'\notin T(C_i\cap S_k^0)$ for at least one of $C_1, C_2$.

Otherwise, choose two different point $x_1$ and $x_2$ in the opposite side and form $C_i = (xyx_i)$. Exact same argument shows that at least one of the two circles will satisfy the desired property.
\end{proof}

Now let $y_1,y_2$ be the intersection point of the circle $C$ with $S_k^0$, and $y_1',y_2'$ be its image respectively. Since $T$ is weakly circle preserving and $y',y_1',y_2'$ determines a circle (as no two of them are equal), we have $x'$ lies on the circle $(y'y_1'y_2')\subset T(S_k^0)$ which is a contradiction to $x'\notin T(S_k^0)$.\\

Case $m = 2$:

(1): Given $S_{k+2}$ containing $S_k^0$, and suppose for contradiction that $|T(S_{k+2})-T(S_k^0)|\geq 3$. Choose $x_1',x_2',x_3' \in T(S_{k+2})-T(S_k^0)$ and fix $x_1,x_2,x_3$ in its preimage, i.e. $T(x_i) = x_i'$, and let $S_{k+1,i} := (S_k^0,x_i)$ for $i=1,2,3$. Notice that $\dim(S_{k+1,i}) = k+1$ as $x_i \notin S_k^0$, so by the (2) of $m=1$, we have $T(S_{k+1,i}-S_k^0) = \{x_i'\}$. Now given $y\in S_k^0$, and $y'$ be its image under $T$.

\begin{cl2}
There is a circle $C_y$ though $y$ and intersect $S_{k+1,i}-S_k^0$ for all $i = 1,2,3$.
\end{cl2}
\begin{proof} (of Claim 2)

\begin{figure}[h!]
\centering
\includegraphics[width=70mm]{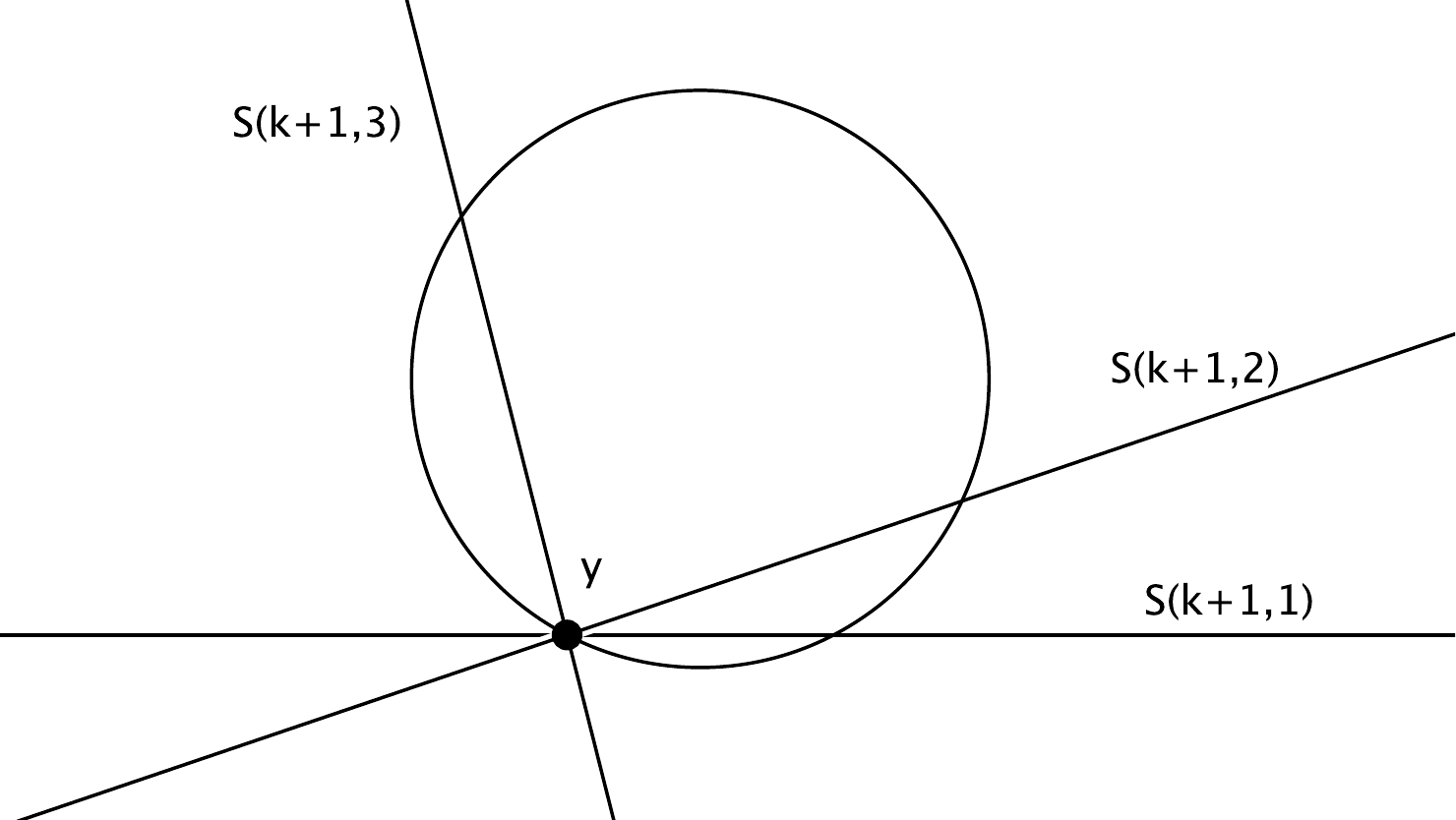}
\caption{}
\label{fig1}
\end{figure}

Under a \m transformation, we may assume that $S_k^0 = \mathbb{R}^k _\infty \subset \R^n_\infty$ and $S_{k+2} =\mathbb{R}^{k+2}_\infty\subset \R^n_\infty$ and $y = \vec{0}$. Then we have $S_{k+1,i}$ is the union of the point $\infty$ and a (k+1)-dimensional plane in $\mathbb{R}^{k+2}$ containing $\mathbb{R}^k$. If we consider the 2-dimensional plane $P$ spanned by $e_{k+1}$ and $e_{k+2}$ where $e_i$ is the usual basis in $\mathbb{R}^n$, then $S_{k+1,i}\cap P$ is a line passing through $\vec{0}$. It is clear that we can choose a circle $C$ in $P$ through $\vec{0}$ that not tangent to the three lines $S_{k+1,i}\cap P$ ($i=1,2,3$). This circle will intersect $S_{k+1,i}\cap P$ at another point other then $\vec 0$, which is not in $S_k^0$, so we proved the claim (see Figure~\ref{fig1}).
\end{proof}

Since $T(S_{k+1,i}-S_k^0) = \{x_i'\}$, we have $\{x_1',x_2',x_3'\}\subset T(C_y)$. Since $T$ is weakly circle preserving, $T(C_y)\subset (x_1'x_2'x_3')$, in particular, $y' := T(y)\in (x_1'x_2'x_3')$. This is true for any $y\in S_k^0$, so $T(S_k^0) \subset (x_1'x_2'x_3')$ which is a contradiction.\\

(2): Let $T(S_{k+2}) - T(S_k^0) = \{x_1',x_2'\}$. Fix $x_1,x_2$ in its preimage and let $S_{k+1,i} := (S_k^0,x_i)$. Again, $\dim(S_{k+1,i}) = k+1$, and $T(S_{k+1,i} - S_k^0) = \{x_i'\}$.

Suppose for contradiction that there exists $y\in S_{k+2}-S_k^0$ such that $y' := T(y) \in T(S_k^0)$, let $S_{k+1,y} = (S_k^0,y)$. Notice that it is necessary that $T(S_{k+1,y}) = T(S_k^0)$.

\begin{cl2}
$T(S_{k+1,y} - S_k^0) = \{y'\}$.
\end{cl2}
\begin{proof} (of Claim 3)

Again, under a \m transformation, we may assume that $S_k^0 = \mathbb{R}^k_\infty\subset\R^n_\infty$ and $S_{k+2} =\mathbb{R}^{k+2}_\infty\subset\R^n_\infty$.

Suppose for contradiction, then there exists $y_1,y_2$ on the opposite side of $S_k^0$ in $S_{k+1,y}$ such that $T(y_1)\neq T(y_2)$. Let $C$ be the circle through $y_1$, $y_2$ and $x_1$.

\begin{figure}[h!]
\centering
\includegraphics[width=70mm]{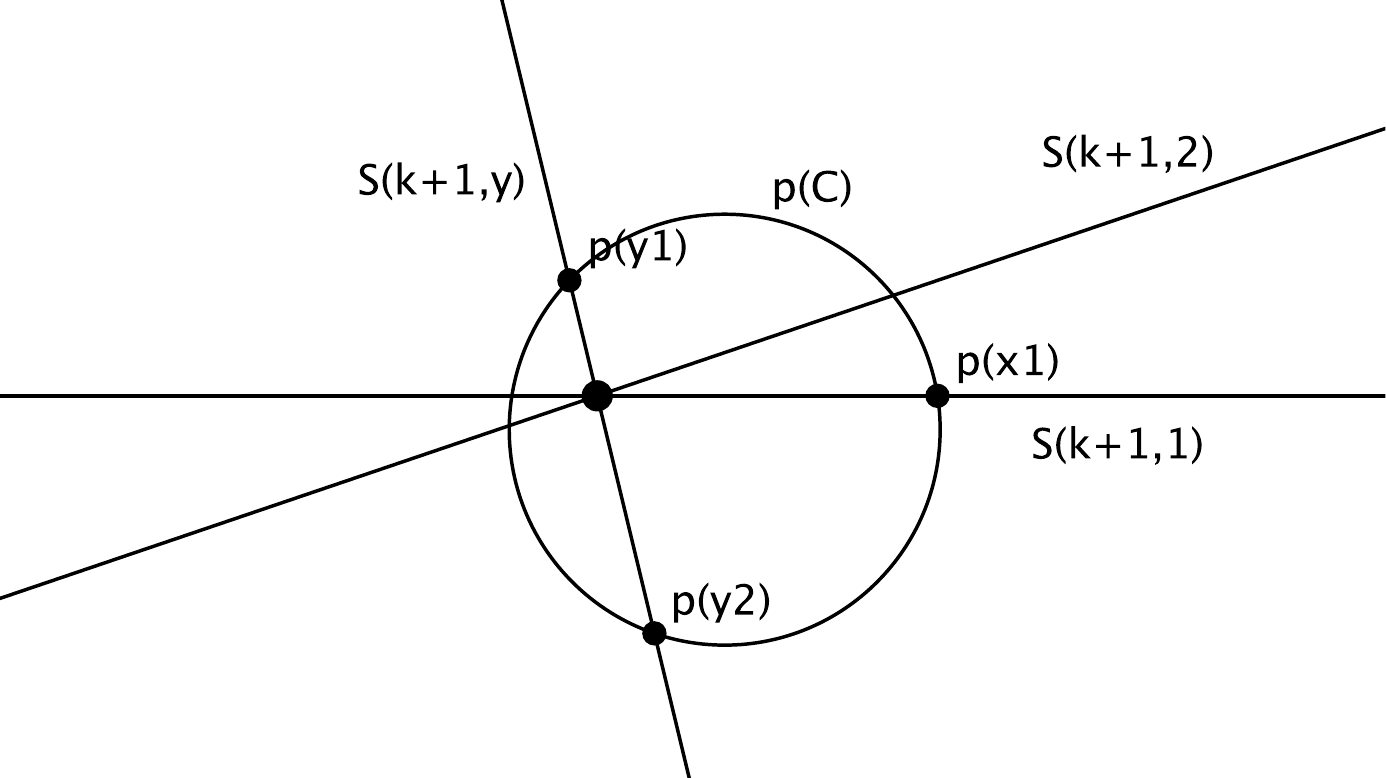}
\caption{}
\label{fig2}
\end{figure}

Now consider the projection map $p:\mathbb{R}^{k+2} \longrightarrow \mathbb{R}^2$ defined by $p(x_1,...,x_{k+2}) = (x_{k+1},x_{k+2})$. Notice that $p(S_{k+1,i}-\{\infty\})$ and $p(S_{k+1,y}-\{\infty\})$ are lines in $\mathbb{R}^2$, $p(y_1)$ and $p(y_2)$ are on the opposite side of $\vec{0}$. Since $p(C)$ is a closed simple loop through $p(x_1)$, $p(y_1)$ and $p(y_2)$, we have $p(C)\cap p(S_{k+1,2}-S_k^0) \neq \emptyset$ (see Figure~\ref{fig2}). Therefore, we have $C\cap S_{k+1,2}-S_k^0 \neq \emptyset$. This means $\{x_1',x_2',y_1',y_2'\}\subset T(C)$, so $x_1',x_2',y_1',y_2'$ are on a circle. This is a contradiction to $\dim(S_{k+2}')\geq k+2$ by lemma \ref{circle-in-S2}.
\end{proof}

Now replace $S_{k+1,3}$ by $S_{k+1,y}$ in claim 2, and exact same argument shows that $T(S_k^0) \subset (x_1'x_2'y')$ which is a contradiction.\\

Case $m \geq 3$:

Induction hypothesis: Assume the lemma holds for $1,...,m-1$.

(1): Given $S_{k+m}$ containing $S_k^0$ and $\dim(S_{k+m}')\geq k+m$. Suppose for contradiction that $|T(S_{k+m})-T(S_k^0)|\geq m+1$.

Let $\{x_1',...,x_m'\} \subset T(S_{k+m})-T(S_k^0)$ such that $T(S_k^0)$ and $\{x_1',...,x_m'\}$ determine a $(k+m)$-sphere, then $T(S_k^0)$ and any $l$-subset of $\{x_1',...,x_m'\}$ determine a $(k+l)$-sphere. Let $x_{m+1}'\in T(S_{k+m})-T(S_k^0)-\{x_1',...,x_m'\}$, and choose a point in $\{x_1',...,x_m'\}$ that not in the $(k+1)$-sphere determined by $T(S_k^0)$ and $x_{m+1}'$, say this point is $x_m'$. Fix $x_1,...,x_{m+1}$ in its preimage, and let $S_{k+m-1}:=(S_k^0,x_1,...,x_{m-1})$ and $S_{k+2} := (S_k^0,x_m,x_{m+1})$.

We let $D:= \dim(S_{k+m-1})$, then it is clear that $D \leq k+m-1$, and $\dim(S_{k+m-1}') = k+m-1\geq D$, so by induction hypothesis, we have $m-1\leq |T(S_{k+m-1}) - T(S_k^0)| = D-k\leq m-1$. This implies that $D = k+m-1$ and $T(S_{k+m-1}-S_k^0) = \{x_1',...,x_{m-1}'\}$.

Similarly, we have $\dim(S_{k+2}) = k+2$ and $T(S_{k+2}-S_k^0) = \{x_m',x_{m+1}'\}$.

Notice that $S_k^0\subset S_{k+m-1}\cap S_{k+2}$, so $S_{k+m-1}\cap S_{k+2}$ contains at least two points. By lemma \ref{dim-sphere-intersect}, we know $\dim(S_{k+m-1}\cap S_{k+2}) \geq (k+m-1)+(k+2)-(k+m) = k+1$. This means that there exists $x\in (S_{k+2}-S_k^0)\cap S_{k+m-1}$. Notice that $T(x) \in T(S_{k+2}-S_k^0) = \{x_m',x_{m+1}'\}$. Therefore, $T(S_k+m-1)-T(S_k^0)$ contains at least $m$ points, which is a contradiction.\\

(2): Let $T(S_{k+m}) - T(S_k^0) = \{x_1',...,x_m'\}$, and fix $x_1,...,x_m$ in its preimage, and let $S_{k+m-1} := (S_k^0,x_1,...,x_{m-1})$. Exact same argument in the proof of (1) of case $m\geq 3$, we have $\dim(S_{k+m-1}) = k+m-1$ and $T(S_{k+m-1}-S_k^0) = \{x_1',...,x_{m-1}'\}$.

Suppose for contradiction that there exists $y\in S_{k+m}-S_k^0$ such that $y':=T(y) \in T(S_k^0)$. Let $S_{k+2} := (S_k^0,x_m,y)$. Notice that $k+1\leq \dim(S_{k+2}) \leq k+2$ and $\dim(S_{k+2}')\geq k+1$. So by the induction hypothesis for (2) of $m=1$, we have $\dim(S_{k+2})\neq k+1$. Hence, $\dim(S_{k+2}) = k+2$.

Notice that $S_k^0\subset S_{k+m-1}\cap S_{k+2}$, so $S_{k+m-1}\cap S_{k+2}$ contains at least two points. By lemma \ref{dim-sphere-intersect}, we know $\dim(S_{k+m-1}\cap S_{k+2}) \geq (k+m-1)+(k+2)-(k+m) = k+1$. This means that there exists $x\in (S_{k+m-1}-S_k^0)\cap S_{k+2}$. Notice that $T(x) \in T(S_{k+m-1}-S_k^0) = \{x_1',...,x_{m-1}'\}$, say $T(x) = x_1'$. Then $S_{k+2}$ is a (k+2)-sphere containing $S_k^0$ with $\dim(S_{k+2}')\geq k+2$ (as it contains $T(S_k^0)$ and $x_1', x_m'$), so by induction hypothesis, we have $T(S_{k+2}-S_k^0) = \{x_1',x_m'\}$ which is a contradiction to $y'\in T(S_{k+2}-S_k^0)$.
\end{proof}

\begin{proof} (of the theorem \ref{thm-main})

We will prove the theorem by building up a chain $S_2^0\subset S_3^0\subset ... \subset S_n^0 = \mathbb{S}^n$, where $S_k^0$ is $k$-sphere in $\mathbb{S}^n$, such that $T$ restrict to $S_k^0$ is a \m transformation.

We will build this chain by induction.
Base case $k=2$: By assumption, there is a two sphere $S_2$ with $T(S_2)$ in circular general position, so by five-point theorem (Theorem \ref{5point}), we conclude that $T$ is a \m transformation on $S_2$.

The induction step: Assume that we have build the chain $S_2^0\subset S_3^0\subset ... \subset S_k^0$ ($k\leq n-1$), with $T$ being a \m transformation on each sphere, we will build $(k+1)$-sphere $S_{k+1}^0$. 

Case 1: There is an $(k+1)$-sphere $S_{k+1}$ containing $S_k^0$ such that $T(S_{k+1})-T(S_k^0)$ consists of at least two points. By the lemma \ref{2point-lemma}, we know that $T$ is a \m transformation on $S_{k+1}$, so we let $S_{k+1}^0 = S_{k+1}$.

Case 2: There is no $(k+1)$-sphere $S_{k+1}$ containing $S_k^0$ such that $T(S_{k+1})-T(S_k^0)$ consists of at least two points. In other words, for any $(k+1)$-sphere $S_{k+1}$ containing $S_k$, we have $|T(S_{k+1})-T(S_k^0)|\leq 1$. Then the hypothesis for lemma \ref{1point-lemma} is satisfied, so we conclude that in particular, $|T(\mathbb{S}^n)-T(S_k^0)|=m$ (notice $\dim((\mathbb{S}^n)') = n \geq n$). But on the other hand, $|T(\mathbb{S}^n)-T(S_k^0)|\geq m+1$ as $T(\mathbb{S}^n)$ is in general position. So this is a contradiction, which means case 2 cannot happen.

Therefore, the theorem follows.
\end{proof}

\section{Weakly Sphere-Preserving Maps}
We first show that weakly sphere-preserving maps are automatically weakly circle-preserving maps.

\begin{lem}\label{lem21}
Suppose $T:\mathbb{S}^n\longrightarrow \mathbb{S}^n$ is weakly sphere-preserving, 
and assume that $T(\mathbb{S}^n)$ is not contained in some $(n-1)$-sphere, i.e., $\dim((\mathbb{S}^n)') = n$.
Then $T$ maps $k$-spheres into $k$-spheres for all dimensions $k$ with $1 \le k \le n-1$.
In particular,  $T$ is weakly circle-preserving.
\end{lem}
\begin{proof}
We will prove by downwards induction on $k$ that $T$ maps $k$-spheres into $k$-spheres.
The base case $k=n-1$ is true by the weakly sphere-preserving hypothesis.

For the induction step, 
assume  that $T$ maps $(k+1)$-spheres into $(k+1)$-spheres, for a fixed $k+1 \ge 2$. 
Given a $k$-sphere $S_k$, choose a point $x_1\in \mathbb{S}^n-S_k$, let $S_{k+1,1}:=(S_k,x_1)$, 
then by induction hypothesis, $\dim(S_{k+1,1}')=k+1$ (recall here $(S_k,x_1)$ means the smallest 
dimension sphere containing $S_k$ and $x_1$ and $S_{k+1,1}' = (T(S_{k+1,1}))$). 
Since $T(\mathbb{S}^n)$ is not contained in some $(n-1)$-sphere, there is an image point 
$x_2' \in \mathbb{S}^n-S_{k+1,1}'$. We fix $x_2 \in T^{-1}(x_2')$, and denote $S_{k+1,2}:=(S_k,x_2)$.
 Notice that $S_k\subset S_{k+1,1}\cap S_{k+1,2}$, so that
 \[
 T(S_k) \subset T(S_{k+1,1})\cap T(S_{k+1,2}) \subset S_{k+1,1}'\cap S_{k+1,2}'. 
 \]
 But $S_{k+1,1}'\neq S_{k+1,2}'$ as $x_2'\in S_{k+1,2}' - S_{k+1,1}'$, so $S_{k+1,1}'\cap S_{k+1,2}'$ is a sphere of dimension less than or equal to $k$. This proves that $T(S_k)$ lies in some $k$-sphere, and completes the induction step.
\end{proof}

%
%

\begin{rmk}
Notice that if $T(\mathbb{S}^n)$ is in spherical general position, then 
necessarily $T(\mathbb{S}^n)$ is not contained in any $(n-1)$-sphere. 
It follows that  a weakly sphere-preserving map having image in spherical general position satisfies the hypothesis of 
Lemma \ref{lem21},  hence is a weakly circle preserving map.
\end{rmk}

Now we have the following analogue theorem for weakly sphere-preserving maps.

\begin{theorem}
For $n\geq 3$, $T:\mathbb{S}^n\longrightarrow \mathbb{S}^n$ be a weakly sphere-preserving map. If $T(\mathbb{S}^n)$ is in spherical general position and there is a 2-sphere $S_2$ with $T(S_2)$ in circular general position, then $T$ is a \m transformation.
\end{theorem}
\begin{proof}
This follows immediately from Lemma \ref{lem21} and Theorem \ref{thm-main}
\end{proof}

\section{Acknowledgments}
The first author thanks N. Nygaard (University of Chicago) for helpful discussions. 
The second author worked on this project in an 
REU Summer Program at the University of Michigan with supervisor  J.  C. Lagarias. 
Both authors thank J. C. Lagarias for helpful discussions and for editorial assistance.
We acknowledge use of  the computer package {\em Cinderella 2} \cite{RK12}  to draw the figures in this paper.

\end{document}